\documentclass[12 pt]{amsart}
\usepackage{amssymb,amsmath,amsthm,mathrsfs,multirow,enumerate}
\usepackage{mathtools}
\usepackage{enumitem} 
\usepackage{graphicx} 
\usepackage[all]{xy}
\usepackage{braket}

\usepackage[indentafter]{titlesec}
\titleformat{name=\section}{}{\thetitle.}{0.8em}{\centering\scshape}
\titleformat{name=\subsection}{}{\thetitle.}{0.5em}{\bfseries}
\titlespacing{\subparagraph}{0em}{0em}{0.5em}

\allowdisplaybreaks

\theoremstyle{definition}

\newtheorem*{remark}{Remark}

\newtheorem*{theorem*}{Theorem}
\newtheorem*{definition*}{Definition}

\newtheorem{theorem}{Theorem}[section]
\newtheorem{thm}[theorem]{Theorem}
\newtheorem{lemma}[theorem]{Lemma}
\newtheorem{prop}[theorem]{Proposition}
\newtheorem{cor}[theorem]{Corollary}

\newtheorem{quest}[theorem]{Question}
\newtheorem{ex}[theorem]{Example}
\newtheorem{rmk}[theorem]{Remark}

\newcommand{\spann}{\mathrm{span}}

\newcommand{\diam}{\mathrm{diam}}
\newcommand{\dist}{\mathrm{d}}

\begin{document}

\title[Non-commuting graphs of $P(L/Z(L))$]{Non-commuting graphs of projective spaces over central quotients of Lie algebras}

\author{Songpon Sriwongsa}
\address{Songpon Sriwongsa \\ Department of Mathematics \\ Faculty of Science \\ King Mongkut's University of Technology Thonburi (KMUTT) \\ 126 Pracha-Uthit Road \\ Bang Mod, Thung Khru \\ Bangkok 10140, Thailand}
\email {\tt songpon.sri@kmutt.ac.th; songponsriwongsa@gmail.com}



\begin{abstract}
Let $L$ be a finite-dimensional non-abelian Lie algebra with the center $Z(L)$. In this paper, we define a non-commuting graph associated with $L$ as the graph whose vertex set is the projective space of the quotient algebra $L/Z(L)$,
and two vertices $\spann \{ x + Z(L) \}$ and $\spann \{ y + Z(L) \}$ are adjacent if $x$ and $y$ do not commute under the Lie bracket of $L$. We present several theoretical properties of this graph. For certain classes of Lie algebras, we show that if the non-commuting graphs from two Lie algebras are isomorphic, then these Lie algebras themselves must be isomorphic. Furthermore, we discuss a relation between graph isomorphisms between non-commuting graphs of Lie algebras over finite fields and the size of the algebras.
\end{abstract}

\maketitle

\bigskip

\section{Introduction}


The connection between algebra and graph theory is one of the main focuses of researchers. It leads to many interesting results and questions. In fact, many algebraic structures can be investigated using the properties of graphs, and vice versa. In the study of groups, a graph associated with a group can be constructed in many different ways (see, for example, \cite{B83, BHM90, LW81, N76}). 

For a group $G$ with the center $Z(G)$, the non-commuting graph $\Gamma_G$ whose vertex set is $G \setminus Z(G)$ and two vertices $x$ and $y$ are adjacent if and only if $xy \neq yx$ was first mentioned by P. Erd\"os \cite{N76}. The graph theoretical properties of $\Gamma_G$ were studied in \cite{AAM06}. The authors of this paper also investigated which group properties of two non-abelian groups with the isomorphic non-commuting graphs are identical. The most recent works on the non-commuting graph of a group can be found in \cite{F25, GM23, S24}. In these works, the authors considered the graphs for finite simple groups and nilpotent groups.

In this paper, we consider an analogue of the aforementioned problem in the context of a Lie algebra. 
Let $L$ be a finite-dimensional non-abelian Lie algebra over a field $\mathbb{K}$ and $Z(L)$ a center of $L$. Recently, a non-commuting graph associated with $L$ was introduced in \cite{CES24}. In this graph, the vertex set is $L \setminus Z(L)$, and two distinct vertices are adjacent if and only if they do not commute under the Lie bracket of $L$. The authors also investigated several properties of this graph. 

It is worth noting that for any $x, y \in L \setminus Z(L), 0 \neq \alpha, \beta \in \mathbb{K}$ and $s, t \in Z(L)$
\[
[x, y] \neq 0 \iff [\alpha x + s, \beta y + t] \neq 0.
\]
As a consequence, it is more natural to define a \textit{non-commuting graph} associated with a Lie algebra $L$, denoted by $\Gamma_L$, as the graph whose vertex set $V(\Gamma_L)$ consists of the points of the projective space $P(L / Z(L))$, that is, the one-dimensional subspaces of $L / Z(L)$, and the adjacency condition: two vertices $[x]$ and $[y]$ are adjacent, denoted $[x] \sim [y]$, if $[x, y] \neq 0$, where $[x]$ denotes the one-dimensional space of $L / Z(L)$ spanned by $x + Z(L)$. 

Following the direction of \cite{AAM06}, we investigate the graph properties of $\Gamma_L$ and consider a few questions regarding the isomorphism of two non-commuting graphs. Note that the base field for $L$ is not fixed in this work.
In what follows, we revisit some definitions and important terminologies of graph theory. The reader is referred to \cite{CM09, GR01} for more details.

An \textit{undirected simple graph} is an ordered pair $G = (V, E)$ comprising a vertex set $V$ and an edge set $E$ which is a set of unordered pairs of two distinct vertices in $G$. The {\it order} of $G$ is defined $|V|$.
If $\{x, y \}$ is in $E$, then we say that $x$ and $y$ are \textit{adjacent}, and denoted by $x \sim y$. For a vertex $x$ in $G$, the {\it vertex degree} of $x$ is defined to be the number of vertices adjacent to $x$, and we denote it by $\deg (x)$. 
A {\it path} of length $k$ from a vertex $x$ and a vertex $y$ in $G$ is a sequence of $k + 1$ distinct vertices from $x$ to $y$ such that consecutive vertices are adjacent. A path that ends at the starting vertex is called a \textit{cycle} and the length of the shortest cycle is called the {\it girth} of the graph, denoted by $g(G)$. A {\it triangle} is a cycle of length $3$. 
A graph $G$ is called {\it Hamiltonian} if it has a cycle that visits each vertex exactly once, ending on the same vertex as it began. If $G$ contains a cycle containing all the edges exactly once, then $G$ is called {\it Eulerian}.
The length of the shortest path connecting a vertex $x$ and a vertex $y$ is called the {\it distance} between $x$ and $y$, denoted by $\dist (x, y)$. The {\it diameter} of $G$ is defined as $\diam (G) :=  \displaystyle \max_{x, y \in V} \dist (x, y)$. A graph $G$ is said to be {\it connected} if for any two distinct vertices, there always exists a path between them. 
If all distinct vertices in $G$ are adjacent to each other, then $G$ is said to be {\it complete}. A subgraph of $G$ that is complete is called a {\it clique}.  
The number of vertices in the largest clique in $G$ is called the {\it clique number} of $G$, denoted by $\omega (G)$.
The {\it chromatic number} of $G$, denoted by $\chi(G)$, is the minimum number of colors needed to paint on the vertices such that the adjacent vertices get different colors.
 Recall that an \textit{independent} set is a set of vertices in $G$ such that for every two vertices, there is no edge connecting the two. The maximum size of an independent set in $G$ is called the {\it independence number} of $G$ and denoted by $\alpha (G)$.
 If $V$ can be partitioned into $k$ different independent sets, namely $V_1, V_2, \ldots, V_k$, then $G$ is called a $k$-\textit{partite graph}  and $V_i, i = 1, 2, \ldots, k$, is called a {\it partite set}. A {\it complete $k$-partite graph} is a $k$-partite graph in which there is an edge between every pair of vertices from different partite sets. We say $G$ is a {\it complete multipartite graph} if it is a graph that is complete $k$-partite for some $k \in \mathbb{N} \cup \{\infty\}$.
If every vertex in $G$ has an equal degree $k$, then $G$ is called a $k$-{\it regular graph}. 
The {\it vertex connectivity}, denoted by $\kappa(G)$, of a connected graph $G$ is the smallest number of vertices whose removal disconnect $G$. For a connected graph $G$, a subset $S$ of the vertex set $V(G)$ is called a {\it cut set} if deletion of all vertices in $S$ makes $G$ disconnect. Again, for a subset $S$ of the vertex set $V(G)$, let $N_G[S]$ be the set of vertices in $G$ which are in $S$ or adjacent to a vertex in $S$. Moreover, the set $S$ is said to be a {\it dominating set} if $N_G[S] = V(G)$. The {\it domination number} of a graph $G$, denoted by $\gamma(G)$, is the minimun size of a dominating set of the vertices in $G$.
Finally, a graph $G$ is said to be {\it planar} if it can be embedded in the plane. 

\medskip

{\bf Organization of this paper:} All Lie algebras in this paper are assumed to be finite-dimensional and non-abelian. In Section \ref{section2}, we analyze several graph theoretical properties of our non-commuting graphs. It will be seen that both Lie algebra structures and base fields can affect the graph properties. Finally, we consider questions of what we can say about two Lie algebras whose non-commuting graphs are isomorphic; see Question \ref{isomq1} and Question \ref{isomq2} in Section \ref{ISO}. To simplify the problems, we consider three different classes of Lie algebras: CT Lie algebras, AC Lie algebras, and Lie algebras over finite fields.

\section{Basic properties of non-commuting graphs}\label{section2}

In this section, we investigate several graph properties of the non-commuting graph $\Gamma_L$ associated with a finite-dimensional non-abelian Lie algebra $L$ over a field $\mathbb{K}$. Numerous properties can be derived from a finite graph. Recall that for any subset $S \subseteq L$, the centralizer of $S$ in $L$, denoted by $C_L(S)$, is the subalgebra of $L$ whose elements commute with all elements in $S$, i.e.,
$
C_L(S) = \{y \in L \mid [y, s] = 0 \text{ for all } s \in S \}
$. In particular, for any $x \in L$, 
$
C_L(x) = \{y \in L \mid [y, x] = 0 \}
$.

We observe that for any vertices $[x]$ and $[y]$ of $\Gamma_L$, 
\begin{align*}
    [x] \sim [y] &\iff y \notin C_L(x) \\
    &\iff [y] \notin P(C_L(x) / Z(L)).
\end{align*}

\begin{prop} \label{basic}
    The following statements hold for the graph $\Gamma_L$.
    \begin{enumerate}
        \item $\Gamma_L$ is finite if and only if $\dim(L / Z(L)) = 1$ or
        $\mathbb{K}$ is a finite field.
        \item If $\mathbb{K}$ is finite, then $\Gamma_L$ is regular if and only if for all $x, y \in L \setminus Z(L)$, $\dim (C_L(x)) = \dim (C_L(y))$.
    \end{enumerate}
\end{prop}
\begin{proof}
    To prove (1), 
    assume that $\dim(L / Z(L)) = n \geq 2$ and $\mathbb{K}$ is infinite. Note that $L / Z(L) \cong \mathbb{K}^n$ as a vector space. Moreover, it is well-known that the projective space $P(\mathbb{K}^n)$ is infinite. Thus, the graph $\Gamma_L$ is also infinite. The converse is clear. 

    Next, assume that $\mathbb{K}$ is finite. Then $V(\Gamma_L)$ is also finite. If $\Gamma_L$ is regular, then for any $[x], [y] \in P(L / Z(L))$, $\deg([x]) = \deg([y])$, that is,
    \[
    | \{[a] \in V(\Gamma_L) \mid [a, x] \neq 0 \} | = | \{[a] \in V(\Gamma_L) \mid [a, y] \neq 0 \} |.
    \]
    This implies $\Big | C_L(x) / Z(L) \Big | = \Big | C_L(y) / Z(L) \Big |$. Hence, $\dim (C_L(x)) = \dim (C_L(y))$. The converse can be shown similarly.
\end{proof}

\begin{prop}
$\Gamma_L$ is a complete graph if and only if for all $x \in L \setminus Z(L)$, $C_L(x) = \langle x \rangle \oplus Z(L)$.
\end{prop}
\begin{proof}
Assume that $\Gamma_L$ is complete. Let $x \in L \setminus Z(L)$ and $y \in C_L(x)$. If $y \notin \langle x \rangle \oplus Z(L)$, then $[y] \sim [x]$, and so $[y, x] \neq 0$, a contradiction. Thus, $C_L(x) = \langle x \rangle \oplus Z(L)$.

 Conversely, suppose that for all $x \in L \setminus Z(L)$, $C_L(x) = \langle x \rangle \oplus Z(L)$. Let $[y], [z]$ be two distinct vertices in $V(\Gamma_L)$. Then $C_L(y) = \langle y \rangle \oplus Z(L)$, which implies that $z \notin C_L(y)$. Hence, $[y] \sim [z]$.
\end{proof}

\begin{thm}\label{diameter}
Assume that $V(\Gamma_L)$ contains at least two elements.
Then $\Gamma_L$ is connected and its diameter is $\diam (\Gamma_L) = 1$ or $2$. Moreover, the girth of $\Gamma_L$ equals $3$.
\end{thm}
\begin{proof}
    Since $L$ is non-abelian, there exist $a, b \in L \setminus Z(L)$ such that $[a, b] \neq 0$. Then $[a] \sim [b]$ and so $\diam (\Gamma_L) \geq 1$. Let $[x]$ and $[y]$ be two distinct non-adjacent vertices in $\Gamma_L$. Then $x, y \in L \setminus Z(L)$ and so
    there exist $x', y' \in L \setminus Z(L)$ such that $[x, x'] \neq 0$ and $[y, y'] \neq 0$, i.e., 
    $[x] \sim [x']$ and $[y] \sim [y']$. If $[x', y] \neq 0$ or $[x, y'] \neq 0$, then
    $[y] \sim [x']$ or $[x] \sim [y']$, and so $\dist ([x], [y]) = 2$. Now suppose otherwise, we have $[y, x'] = [x, y'] = 0$. So, $[x, x' + y'] = [x, x'] + [x, y'] = [x, x'] \neq 0$ and $[y, x' + y'] = [y, x'] + [y, y'] = [y, y'] \neq 0$. Thus, $[x' + y']$ is a vertex in $\Gamma_L$ which is adjacent to both $[x]$ and $[y]$, so we have $\dist ([x], [y]) = 2$. Hence, $\diam (\Gamma_L) = 1$ or $2$. From the same arguments, we see that $\Gamma_L$ is connected. Moreover, it is obvious that for any two adjacent vertices $[x]$ and $[y]$, we have $[x + y]$ is a vertex in $V(\Gamma_L)$ which is distinct from $[x]$ and $[y]$ and $[x] \sim [x + y] \sim [y] \sim [x]$ which is a triangle. So, the girth of $\Gamma_L$ is $3$.
\end{proof}

When the base field is finite, we have the following two theorems.

\begin{thm}\label{Hamilton}
    Let $L$ be a finite-dimensional non-abelian Lie algebra over a finite field $\mathbb{F}_q$. Then the graph $\Gamma_L$ is Hamiltonian. 
\end{thm}
\begin{proof}
    Let $[x] \in V(\Gamma_L)$. Then 
    \begin{align*}
        \deg ([x]) &= |P(L / Z(L)) \setminus P(C_L(x) / Z(L))| \\
        & = |P(L / Z(L))| - |P(C_L(x) / Z(L))|.
    \end{align*}
    Suppose that $\dim (L) = n$, $\dim (C_L(x)) = r$ and $\dim (Z(L)) = s$. Then $s < r < n$ because $x \in L \setminus Z(L)$. Moreover, $|V(\Gamma_L)| = \dfrac{q^{n - s} - 1}{q - 1}$ and $\deg ([x]) = \dfrac{q^{n - s} - q^{r - s}}{q - 1}$. This implies that 
    $\deg(x) >  \dfrac{|V(\Gamma_L)|}{2}$ since $q \geq 2$.
    Hence, by Dirac's theorem \cite{Dirac52}, $\Gamma_L$ is Hamiltonian.
\end{proof}

\begin{thm}\label{Eulerian}
    Let $L$ be a finite-dimensional non-abelian Lie algebra over a finite field $\mathbb{F}_q$. If either of the following statements hold:
    \begin{enumerate}
        \item $q = 2$, or
        \item $q > 2$ and for all vertex $[x] \in V(\Gamma_L)$, $C_L(x)$ has even codimension, 
    \end{enumerate}
    then $\Gamma_L$ is Eulerian.
\end{thm}
\begin{proof}
    As in the proof of Theorem \ref{Hamilton}, for all vertex $[x] \in V(\Gamma_L)$, we have $\deg ([x]) = q^{r - s}(1 + q + \cdots + q^{n - r - 1})$. Under the assumptions, it can be seen that every vertex of $\Gamma_L$ has an even degree. By Euler's theorem, this graph is Eulerian.
\end{proof}

\begin{thm}\label{planar}
    Let $L$ be a finite-dimensional non-abelian Lie algebra over a finite field $\mathbb{F}_q$. Then $\Gamma_L$ is planar if and only if one of the following conditions holds:
    \begin{enumerate}
        \item $L$ is a $2$-step nilpotent Lie algebra over $\mathbb{F}_q$ where $q = 2$ or $3$ with abelian $2$-dimensional $L/Z(L)$,
        \item $L$ is a non-abelian Lie algebra over $\mathbb{F}_q$ where $q = 2$ or $3$ with non-abelian 2-dimensional $L/Z(L)$.
    \end{enumerate}
\end{thm}
\begin{proof}
    Assume that $L$ is a non-abelian Lie algbra over $\mathbb{F}_q$ of dimension $n > 1$. Suppose that $\Gamma_L$ is planar. By \cite[Theorem 10.1.5]{CM09}, there exists $[x] \in V(\Gamma_L)$ such that $\deg ([x]) \leq 5$. Since 
    $$
    \deg ([x]) = q^{r - s}(1 + q + \cdots + q^{n - r - 1}),
    $$
    where $n = \dim(L), r = \dim (C_L(x)), s = \dim (Z(L))$ with $s < r < n$, we have $q \leq 5$. Moreover, $\dim (L / Z(L)) = n - s = 2$ or $3$ when $q = 2$
    and $\dim (L / Z(L)) = n - s = 2$ when $q = 3, 4$ or $5$.
    First, we consider when $\dim (L/Z(L)) = 2$. 
    By \cite[Theorem 3.1]{EW06}, there are two cases:
    \begin{enumerate}
        \item $L/Z(L)$ is abelian with a basis $\{x + Z(L), y + Z(L)\}$ and $ 0 \neq [x, y] \in Z(L)$ (because $L$ is non-abelian), 
        \item $L/Z(L)$ has a basis $\{x + Z(L), y + Z(L)\}$ with the condition $[x + Z(L), y + Z(L)] = x + Z(L)$.
    \end{enumerate}
    For both cases, if $q = 4$ and $\mathbb{F}_4 = \{0, 1, \alpha, \alpha + 1\}$, where $\alpha^2 = \alpha + 1$, then $$
    P(L/Z(L)) = \{[x], [y], [x + y], [\alpha x + y], [(\alpha + 1)x + y] \}.
    $$
    Thus, $\Gamma_L$ would be the complete graph $K_5$, which is non-planar. Similarly, if $q = 5$, $\Gamma_L$ would be the complete graph $K_6$, which is also non-planar. Note that case (1) implies that $L$ is a $2$-step nilpotent Lie algebra over $\mathbb{F}_q$ where $q = 2$ or $3$ with abelian $2$-dimensional $L/Z(L)$. On the other hand, case (2) gives $L$ a non-abelian Lie algebra over $\mathbb{F}_q$ where $q = 2$ or $3$ with non-abelian $2$-dimensional $L/Z(L)$.
        Now, assume that $\dim (L/Z(L)) = 3$ and $q = 2$. Suppose that $L/Z(L)$ has a basis $\{x + Z(L), y + Z(L), z + Z(L)\}$ with the bracket relations in $L$: $[x, y] = a_1, [x, z] = a_2$ and $[y, z] = a_3$ for some $a_1, a_2, a_3 \in L$. Then the vertex set $V(\Gamma_L)$ is
    \[
        P(L/Z(L)) = \{[x], [y], [z], [x + y], [x + z], [y + z], [x + y + z] \}.
    \]
    It is worth noting that $a_i$'s cannot all be the same; otherwise $x + y + z$ would be a central element in $L$. Suppose that two of $a_i$'s are zero. Without loss of generality, we may assume $a_2 = a_3 = 0$. Then $z$ would be a central element in $L$, a contradiction. Similarly, if one of $a_i$'s is zero, then the rest must be different. Consider a minimal case $a_1 \neq a_2 \in L \setminus \{0\}$ and $a_3 = 0$. The graph $\Gamma_L$ has $7$ vertices and $18$ edges. It is non-planar as a consequence of Euler's formula (see \cite[Theorem 10.1.2]{CM09}). For other cases, there are at least $18$ edges, so again they cannot be planar. Hence, it is impossible that $\dim(L/Z(L)) = 3$ and $q = 2$.

     Conversely, suppose that $q = 2$ and $L/Z(L)$ has a basis $\{x + Z(L), y + Z(L)\}$ with $[x, y]  \neq 0$. Then $P(L/Z(L)) = \{[x], [y], [x + y]\}$ and $\Gamma_L$ is a triangle, so it is planar. Finally, if $q = 3$ and $L/Z(L)$ has a basis $\{x + Z(L), y + Z(L)\}$ with $[x, y]  \neq 0$, then $P(L/Z(L)) = \{[x], [y], [x + y], [x + 2y]\}$ and $\Gamma_L$ is the complete graph $K_4$ which is planar.
\end{proof}

\begin{thm}
    If $L$ is a finite-dimensional non-abelian Lie algebra over a finite field, then $\kappa(\Gamma_L) > 1$.
\end{thm}
\begin{proof}
    Let $S$ be a cut set such that $|S| = \kappa(\Gamma_L)$. Suppose that $\kappa(\Gamma_L)  = 1$. Then $S = \{[x]\}$ for some $[x] \in V(\Gamma_L)$. Let $[u]$ and $[v]$ be two elements that belong to distinct connected components of $\Gamma_L \setminus S$ and $[u] \sim [x] \sim [v]$ in $\Gamma_L$. Then $[u, v] = 0$, and $[u + x]$ is a vertex in $V(\Gamma_L)$ distinct from $[x], [u]$ and $[v]$. In addition, we have $[u + x]  \sim [u]$ and $[u + x] \sim [v]$. This is a contradiction because $[u + x] \notin S$. Hence $\kappa(\Gamma_L) > 1$.
\end{proof}

For any subset $M \subseteq V(\Gamma_L)$, a set of representative elements for $M$ is defined as a set $\mathcal{P}(M)$ that contains exactly one element $x \in L$ from each $[x] \in P(L/Z(L))$. Moreover, we let $U(M)$ denote the union of one dimensional subspaces spanned by each $x \in \mathcal{P}(M)$, that is, 
$$
U(M) = \bigcup\limits_{x_\alpha \in \mathcal{P}(M)} \spann \{x_\alpha\}.
$$
We have the following proposition. 

\begin{prop}
    Let $M$ be a maximal independent set of $\Gamma_L$.
    Then $U(M) \cup Z(L)$ is a maximal abelian subalgebra of $L$.
\end{prop}
\begin{proof}
    Suppose that $M$ is a maximal independent set of $\Gamma_L$.
    Let $x, y \in U(M) \cup Z(L)$. Since $x + y$ commutes with every element of $U(M)$ and $M$ is a maximal independent set, $x + y \in U(M) \cup Z(L)$. Moreover, it is clear that $[x, y] = 0$. Thus $U(M) \cup Z(L)$ is an abelian subalgebra of $L$. 

    Next, we show that $U(M) \cup Z(L)$ is maximal. Let $T$ be an abelian subalgebra of $L$ containing $U(M) \cup Z(L)$. Let $a \in T \setminus Z(L)$. Note that $[a, s] = 0$ for all $s \in U(M)$. Thus $M \cup \{ [a]\}$ is an independent set. By the maximality of $M$, $[a] \in M$, that is, $[a] = [x_\alpha]$ for some $x_\alpha \in \mathcal{P}(M)$. Then $a - \lambda x_\alpha \in Z(L)$ for some $\lambda \neq 0$ and so, $a = \lambda x_\alpha + z$ for some $z \in Z(L)$. By the above paragraph, $a \in U(M) \cup Z(L)$.
    Hence, $T = U(M) \cup Z(L)$. This completes the proof.
\end{proof}


The following theorem presents a property of a nilpotent Lie algebra over a finite field when its non-commuting graph is regular.

\begin{thm}
    Let $L$ be a finite-dimensional non-abelian nilpotent Lie algebra over a finite field. If $\Gamma_L$ is regular, then the nilpotence class of $L$ is at most $3$.
\end{thm}
\begin{proof}
    Suppose that $L$ is nilpotent and $\Gamma_L$ is regular. By Proposition \ref{basic} (2), we have  $|C_L(x)| = |C_L(y)|$ for any two $x, y \in L \setminus Z(L)$. This implies that $L$ has only two centralizer dimensions. Therefore, $L$ has the nilpotence class at most $3$, by \cite[Theorem B]{BI03}.
\end{proof}

\begin{remark}
    The classification of Lie algebras of nilpotence class $3$ is given in \cite{KNS23}.
\end{remark}

The dual question of Erd\"os,as posed by A. Abdollahi et al. \cite{AAM06} in the context of the non-commuting graph of a group is as follows. 

\begin{quest}\label{qind}
    Let $G$ be a group whose non-commuting graph has no infinite independent sets. Is it the independence number of $G$ is finite?
\end{quest}

Some cases of this question have been considered and positively answered in \cite{AAM06}.
The analogous version of Question \ref{qind} for a Lie algebra $L$ can also be considered, and we respond to it without assuming any additional conditions on $L$.

\begin{thm}
    Suppose that $\Gamma_L$ has no infinite independent sets. Then  $\alpha(\Gamma_L) < \infty$. 
\end{thm}
\begin{proof}
    Consider the following two cases. First, if $\Gamma_L$ is complete, then $\alpha(\Gamma_L) = 1$. Next, we assume that there exist two non-adjacent vertices $[x], [y] \in V(\Gamma_L)$. Then $[x, y] = 0$ and $[x + \alpha y] \in V(\Gamma_L)$ for any $\alpha \in \mathbb{K}$ since $\{ x + Z(L), y + Z(L) \}$ is linearly independent in $L/Z(L)$. If we let $S = \{[x + \alpha y] \in V(\Gamma_L) \mid \alpha  \in \mathbb{K}\}$, then $S$ is an independent set and by assumption it must be finite. This implies that $\mathbb{K}$ is a finite field and the result is obtained.
\end{proof}

The following discussion presents results concerning a dominating set of $\Gamma_L$.

\begin{thm}\label{domi1}
      Let $L$ be a finite-dimensional non-abelian Lie algebra and
      $[x]$ a vertex in $\Gamma_L$. Then
     $\{[x]\}$ is a dominating set for $\Gamma_L$ if and only if $ C_L(x) = \spann \{x\} \oplus Z(L)$. 
\end{thm}
\begin{proof}
    Assume that $\{[x]\}$ is a dominating set for $\Gamma_L$. Let $y \in C_L(x) \setminus Z(L)$. Then $[x, y] = 0$ and $[y]$ is a vertex in $\Gamma_L$. The assumption forces $[x] = [y]$. Thus, $y - \lambda x \in Z(L)$ for some $\lambda \in \mathbb{K} \setminus \{0\}$, that is, $y \in \spann \{x\} \oplus Z(L)$, consequently, $C_L(x) = \spann \{x\} \oplus Z(L)$.

    On the other hand, suppose that $C_L(x) = \spann \{x\} \oplus Z(L)$. Let $[y] \in V(\Gamma_L) \setminus \{[x]\}$. Then $y \notin \spann \{x\} \oplus Z(L)$, and so it does not belong to $C_L(x)$. Thus, $[x] \sim [y]$, and this completes the proof.
\end{proof}

\begin{remark}
    If $L$ is a $2$-dimensional non-abelian Lie algebra over $\mathbb{Z}_2$, then the ``only if" part of Theorem \ref{domi1} holds. Thus, $\gamma(\Gamma_L) = 1$. In fact, it is also clear from the fact that $\Gamma_L$ is a triangle.
\end{remark}

The following theorem provides a criterion for a subset of $V(\Gamma_L)$ to be a dominating set of $\Gamma_L$. 

\begin{thm}\label{domi}
A subset $S \subseteq V(\Gamma_L)$ is a dominating set if and only if 
\[
C_L(\mathcal{P}(S)) \setminus Z(L) \subseteq  \bigcup\limits_{x_\alpha \in \mathcal{P}(S)} ( \spann \{x_\alpha\} \oplus Z(L)).
\]
\end{thm}
\begin{proof}
    Suppose that $S \subseteq V(\Gamma_L)$ is a dominating set, and let $y \in C_L(\mathcal{P}(S)) \setminus Z(L)$. Then $[y] \in V(\Gamma_L)$ and $[x_\alpha, y] = 0$ for all $x_\alpha \in \mathcal{P}(S)$. Since $S$ is a dominating set, $[y] = [x_\beta]$ for some $x_\beta \in \mathcal{P}(S)$, which implies $y - \lambda x_\beta \in Z(L)$ for some $\lambda \in \mathbb{K} \setminus \{0\}$. Thus, $y \in \spann \{x_\beta\} \oplus Z(L)$.
    
    Conversely, assume that $C_L(\mathcal{P}(S)) \setminus Z(L) \subseteq  \bigcup\limits_{\alpha \in I} ( \spann \{x_\alpha\} \oplus Z(L))$. Let $[y] \in V(\Gamma_L) \setminus S$. Then $[y] \neq [x_\alpha]$ for all $x_\alpha \in \mathcal{P}(S)$, and so $y \notin \cup_{x_\alpha \in \mathcal{P}(S)} ( \spann \{x_\alpha\} \oplus Z(L))$. Thus, $y \notin C_L(\mathcal{P}(S)) \setminus Z(L)$. Therefore, $[y]$ is adjacent to at least one element of $S$. Therefore, $S$ is a dominant set of $\Gamma_L$.
\end{proof}

Theorem \ref{domi} implies the following corollary for a semisimple Lie algebra. 

\begin{cor}\label{Cartan}
    If $L$ is a finite-dimensional semisimple Lie algebra over an algebraically closed field of characteristic $0$ with a Cartan subalgebra $H$, then the projective space $P(H)$ is a dominating set of $\Gamma_L$. 
\end{cor}
\begin{proof}
    It is well-known that $H = C_L(H)$ and $Z(L)$ is trivial (see \cite{H12}). Note that $C_L( \mathcal{P}(P(H))) = C_L(H)$. Thus, 
    $$
    C_L( \mathcal{P}(P(H))) = H \subseteq \bigcup\limits_{x \in \mathcal{P}(P(H))} \spann \{x\}.
    $$
   By Theorem \ref{domi}, the proof is complete.
\end{proof}

\begin{thm}\label{finitedomi}
    Let $X$ be a basis for $L$. Then $S = \{[x] \mid x \in X \setminus Z(L)\}$ is a dominating set for $\Gamma_L$. Consequently, $L$ always contains a dominating set of $\Gamma_L$ generating a non-abelian subalgebra and  $\gamma (\Gamma_L) \leq |X \setminus Z(L)| < \infty$.
\end{thm}
\begin{proof}
    Since $X$ is linear independent, we can take $\mathcal{P}(S) = X \setminus Z(L)$. Thus, $C_L(\mathcal{P}(S)) = Z(L)$.
    By Theorem \ref{domi}, $S$ is a dominating set for $\Gamma_L$. Moreover, $\spann (\mathcal{P}(S))$ is a non-abelian subalgebra and $\gamma (\Gamma_L) \leq |S| =  |X \setminus Z(L)| < \infty$.
\end{proof}

\begin{cor}
    Let $L$ be a finite dimensional non-abelian Lie algebra with center of codimension two and $\dim (C_L(x) / Z(L)) > 1$ for all $x \in L \setminus Z(L)$. Then $\gamma (\Gamma_L) = 2$.
\end{cor}
\begin{proof}
    By Theorem \ref{domi1}, $\gamma (\Gamma_L) \geq 2$. Let $X$ be a basis for $L$. From Theorem \ref{finitedomi}, $\gamma (\Gamma_L) \leq |X \setminus Z(L)| \leq 2$. This completes the proof.
\end{proof}

\begin{remark}
    The upper bound of Theorem \ref{finitedomi} may not be sharp for some Lie algebras. For example, if $L = \mathfrak{sl}_2(\mathbb{C})$ with the standard basis $X = \{x, y, h\}$, where the Lie bracket relations are: 
    \[
    [h, x] = 2x, \ \ [h, y] = -2y, \ \ [x, y] = h,
    \] then $|X \setminus Z(L)| = 3$. However, if $S = \{ [x + y], [h]\} \subseteq V(\Gamma_L)$, then by taking $\mathcal{P}(S) = \{x + y, h\}$, the centralizer $C_L(\mathcal{P}(S))$ is trivia. By Theorem \ref{domi}, $S$ is a dominating set for $\Gamma_L$, and so $\gamma(\Gamma_L) \leq 2$.
    In fact, $\gamma(\Gamma_L) = 1$ by Theorem \ref{domi1} because $\dim (C_L(x)) = 1$. More generally, it can be concluded by Corollary \ref{Cartan} that for a finite dimensional semisimple Lie algebra $L$ over $\mathbb{C}$ of rank one, $\gamma(\Gamma_L) = 1$.
\end{remark}

We now discuss the domination number in the context of a general semisimple Lie algebra over $\mathbb{C}$.

\begin{cor}
    If $L$ is a finite-dimensional semisimple Lie algebra over $\mathbb{C}$ of rank $t > 1$, then $\gamma(\Gamma_L) = 2$.
\end{cor}
\begin{proof}
    Let $L$ be a finite-dimensional semisimple Lie algebra over $\mathbb{C}$ of rank $t > 1$.
    It is well-known that $\dim (C_L(x)) \geq t$ for all $x \in L \setminus \{0\}$ and $Z(L)$ is trivial. 
    By Theorem \ref{domi1}, $\gamma(\Gamma_L) \geq 2$. Now, consider the Cartan decomposition \cite{H12}:
\[
L = H \oplus \bigoplus_{\alpha \in \Phi}L_\alpha,
\]
where $H$ is a Cartan subalgebra of $L$ and $\Phi$ is a set of roots relative to $H$. Suppose that $\{h_1, h_2, \ldots, h_t \}$ is a basis of $H$ and $L_\alpha = \spann \{ x_\alpha \}$ for each $\alpha \in \Phi$. Let 
\[
S = \left \{ \left[ \sum_{i = 1}^t h_i \right], \ \  \left[\sum_{\alpha \in \Phi} x_\alpha \right] \right \}.
\]
Let $a = \sum_{j = 1}^t c_j h_j + \sum_{\beta \in \Phi} b_\beta x_\beta$ be any element in $L \setminus Z(L)$. So, $c_j$'s and $b_\beta$'s are not all zero. If there exists $b_{\beta_0} \neq 0$, we can choose $h_{i_0}$ such that $\beta_0 (h_{i_0}) \neq 0$. This deduces that $[a, \sum_{i = 1}^t h_i] \neq 0$, that is, $[a] \sim [\sum_{i = 1}^t h_i]$, since $\{x_\beta \mid \beta \in \Phi \}$ is linearly independent. Finally, suppose that $a = \sum_{j = 1}^t c_j h_j$ and $c_{i_0} \neq 0$. Then there must be $x_{\alpha_0}$ such that $[h_{i_0}, x_{\alpha_0}] \neq 0$, otherwise $h_{i_0}$ would be in $Z(L) = \{ 0 \}$. This implies $[a, \sum_{\alpha \in \Phi} x_\alpha] \neq 0$ since $\{x_\beta \mid \beta \in \Phi \}$ is linearly independent. Hence, $[a] \sim [\sum_{\alpha \in \Phi} x_\alpha]$. This completes the proof.
\end{proof}


To end this section, we present a result regarding the chromatic number of $\Gamma_L$.

\begin{thm}\label{chromatic}
     Let $L$ be a finite-dimensional non-abelian Lie algebra over a finite field. Then $\chi (\Gamma_L)$ equals the minimum number of abelian subalgebras of $L$ whose union is $L$.
\end{thm}
\begin{proof}
    Note that a finite abelian subalgebra covering of $L$ always exists, for instance, taking all one-dimensional subalgebras from each element in $L$. Let $k$ be the minimum number of abelian subalgebras of $L$ whose union is $L$ and those algebras are $A_1, A_2, \ldots, A_k$. Then the vertices of $\Gamma_L$ in $P(\spann\{A_i, Z(L)\}/Z(L))$ are not adjacent to each other. This implies $\chi (\Gamma_L) \leq k$. To show the converse, we let $\chi = \chi (\Gamma_L)$. Then there exist $\chi$ independent subsets $M_1, M_2, \ldots, M_\chi$ of vertices of $\Gamma_L$ whose union is $V(\Gamma_L) = P(L \setminus Z(L))$. Recall that for $i = 1, 2, \ldots, \chi$, $\mathcal{P}(M_i) = \{x \in L \mid x \text{ is a representative of } [x] \text{ in } M_i\}$.
    We see that for each $j$, $N_j = \spann (\mathcal{P}(M_j), Z(L))$ is an abelian subalgebra of $L$ and $L = \cup_{i = 1}^\chi N_j$. Hence, $\chi(\Gamma_L) \geq k$.
\end{proof}

\section{Isomorphic non-commuting graphs}\label{ISO}

Recall that an isomorphism $\phi$ of graphs $G$ and $H$ is a bijection between the vertex sets of $G$ and $H$
such that any two vertices $u$ and $v$ of $G$ are adjacent in $G$ if and only if $\phi(u)$ and $\phi(v)$ are adjacent in $H$. If there exists such an isomorphism for $G$ and $H$, we say that $G$ and $H$ are isomorphic and write $G \cong H$. Let $\mathfrak{g}$ and $\mathfrak{h}$ be two finite-dimensional Lie algebras over fields $\mathbb{K}_1$ and $\mathbb{K}_2$, respectively. We note that $\Gamma_{\mathfrak{g}} \cong \Gamma_{\mathfrak{h}}$ if and only if there exists a bijection 
\[
\phi : V(\Gamma_{\mathfrak{g}}) \rightarrow V(\Gamma_{\mathfrak{h}})
\]
such that for all distinct vertices $x, y \in V(\Gamma_{\mathfrak{g}})$, 
\[
[x, y] = 0 \text{ if and only if } [\phi (x), \phi (y)] = 0.
\]

It is clear that for any two Lie algebras $\mathfrak{g}$ and $\mathfrak{h}$ over a field $\mathbb{K}$, if $\mathfrak{g} \cong \mathfrak{h}$ (as Lie algebra), then $\Gamma_{\mathfrak{g}} \cong \Gamma_{\mathfrak{h}}$. One may ask what would happen if $\Gamma_{\mathfrak{g}} \cong \Gamma_{\mathfrak{h}}$ and we state a question here:

\begin{quest}\label{isomq1}
    Let $\mathfrak{g}$ and $\mathfrak{h}$ be two finite-dimensional Lie algebras over a field $\mathbb{K}$. Is it true that $\mathfrak{g} \cong \mathfrak{h}$?
\end{quest}

In general, isomorphic non-commuting graphs do not guarantee the preservation of Lie algebra properties. 

\begin{ex}
    Let $L_1$ be a Lie algebra over $\mathbb{F}_2$ with basis $\{x, y, z\}$, and let $L_2$ be a Lie algebra over $\mathbb{F}_2$ with basis $\{a, b, c\}$, where 
    \[
    [x, y] = z \text{ while } [a, b] = a.
    \]
    Then $\Gamma_{L_1} \cong \Gamma_{L_2}$, and both are $4$-regular graphs on $6$ vertices. However, $L_1$ is nilpotent, whereas $L_2$ is not. Consequently, $L_1$ and $L_2$ are not isomorphic as Lie algebras. 
\end{ex}

We can respond positively to Question \ref{isomq1} for some specific Lie algebras and for the class of CT Lie algebras in Subsection \ref{CT}.

\begin{thm}
    Let $L$ be a finite-dimensional non-abelian Lie algebra over $\mathbb{F}_q, q = 2, 3$ satisfying either
    \begin{enumerate}[label=(\Alph*)]
        \item $L$ is a $2$-step nilpotent with abelian $2$-dimensional $L/Z(L)$, or
        \item the quotient $L/Z(L)$ is $2$-dimensional non-abelian. 
    \end{enumerate}
    If $L$ is a finite-dimensional non-abelian Lie algebra such that $\Gamma_L \cong \Gamma_{\mathfrak{g}}$, then $\mathfrak{g}$ satisfies either (A) or (B).
\end{thm}
\begin{proof}
    The result follows directly from Theorem \ref{planar}.
\end{proof}

\bigskip

\subsection{CT Lie algebras} \label{CT}

A Lie algebra $L$ is called \textit{commutative transitive} (or CT) if for all $x, y, z \in L \setminus \{0\}$, if  
$[x, y] = [y, z] = 0$, then $[x, z] = 0$. Note that $L$ is CT if and only if its nonzero elements have abelian centralizers. It is also obvious that a CT Lie algebra with nontrivial center is abelian. In particular, every nilpotent CT Lie algebra is abelian.
For useful details about CT Lie algebras, the reader can consult \cite{Gor17, KM10}.

\begin{thm}\label{CTcomp}
    A Lie algebra $L$ with $Z(L) = 0$ is CT if and only if $\Gamma_L$ is a complete multipartite graph.
\end{thm}
\begin{proof}
    Let $L$ be a Lie algebra with $Z(L) = 0$. Assume that $L$ is CT. We can define an equivalence relation on $L \setminus \{ 0 \}$ by 
\[
x R y   \iff y \in C_L(x)
\]
for all $x, y \in L \setminus \{ 0 \}$. Since the equivalence class of $x \in L \setminus \{ 0 \}$ is $C_L(x) \setminus \{0\}$, all sets $C_L(x) \setminus \{ 0 \}, x \in L \setminus \{ 0 \}$ form a partition of $L \setminus \{0 \}$. So, all projective spaces $P(C_L(x)), x \in L \setminus Z(L)$, form a partition of $V(\Gamma_L)$. Since each $C_L(x)$ is abelian, $P(C_L(x))$ is a partite set. 
Now, let $[x]$ and $[y]$ be two vertices in $\Gamma_L$ of two different partite sets. Then the first set is $P(C_L(x))$ while the second is $P(C_L(y))$ and we also have
$[x, y] \neq 0$. Let $[a] \in P(C_L(x))$ and $b \in P(C_L(y))$. Then $[a, b] \neq 0$. Indeed, if $[a, b] = 0$, then $[a, y] = 0$ and so $[x, y] = 0$, a contradiction. Thus, $[a] \sim [b]$ and hence $\Gamma_L$ is a complete multipartite graph.

Now we prove the converse. Suppose that $\Gamma_L$ is a complete multipartite graph. For any partite set $P$ and $[x] \in P$, we have $P = P(C_L(x))$. To see this, let $[a] \in P(C_L(x))$, then $[a, x] = 0$, that is, there is no edge between $[a]$ and $[x]$. Since the graph is complete partite, $[a]$ must be in the same partite set with $[x]$. 
Finally, let $x, y, z \in L \setminus \{ 0 \}$ such that $[x, y] = [y, z] = 0$. Then $[x], [y], [z]$ are in the partite set $P(C_L(y))$. This implies $[x, z] = 0$. Therefore, $L$ is CT.
\end{proof}

The following corollaries are consequences of the previous theorem.

\begin{cor}\label{graphCT}
     Let $L$ be a finite-dimensional Lie algebra with trivial center. If $\Gamma_L \cong \Gamma_{\mathfrak{g}}$ where $\mathfrak{g}$ is a finite-dimensional non-abelian CT Lie algebra, then $L$ is CT and hence, directly indecomposable.
\end{cor}
\begin{proof}
    By Theorem \ref{CTcomp}, $\Gamma_L$ is a complete partite graph, and thus, $L$ is CT. The final result is due to \cite[Theorem 2.1]{KM10}.
\end{proof}

From Example 4.4 of \cite{KM10}, we know that every free Lie algebra is CT. Therefore, we have the following corollary.

\begin{cor}
    If $L$ is a free Lie algebra with trivial center, then $\Gamma_L$ is a complete partite graph. 
\end{cor}

When $\mathbb{K}$ is the field of complex numbers $\mathbb{C}$, we have the following result.

\begin{cor}\label{CTsl2}
    Let $L$ be a finite-dimensional semisimple Lie algebra over $\mathbb{C}$. If $\Gamma_L \cong \Gamma_{\mathfrak{sl}_2(\mathbb{C})}$, then $L \cong \mathfrak{sl}_2(\mathbb{C})$.
\end{cor}
\begin{proof}
Note that $\mathfrak{sl}_2(\mathbb{C})$ is CT because the centralizer of any nonzero element in $\mathfrak{sl}_2(\mathbb{C})$ has dimension either $1$ or $2$, so it is abelian.

Assume that $\Gamma_L \cong \Gamma_{\mathfrak{sl}_2(\mathbb{C})}$. By Theorem \ref{CTcomp}, $\Gamma_L$ is complete partite, and so it is CT. It follows from \cite[Theorem 3.1]{KM10} that $L$ is isomorphic to $\mathfrak{sl}_2(\mathbb{C})$.
\end{proof}

More generally, Theorem 4.1 of \cite{KM10} asserts that a finite-dimensional CT Lie algebra over $\mathbb{C}$ is either solvable or simple. The following corollaries holds immediately. 

\begin{cor}
    Let $L$ be a finite-dimensional Lie algebra over $\mathbb{C}$. If $L$ is non-solvable with trivial center and $\Gamma_L \cong \Gamma_{\mathfrak{sl}_2(\mathbb{C})}$, then $L \cong \mathfrak{sl}_2(\mathbb{C})$.
\end{cor}
\begin{proof}
    Since $\Gamma_L \cong \Gamma_{\mathfrak{sl}_2(\mathbb{C})}$, by the similar argument as above, $\Gamma_L$ is complete partite, and so, it is CT. By \cite[Theorem 4.1]{KM10}, $L$ is simple. The proof is complete by Corollary \ref{CTsl2}.
\end{proof}

\begin{cor}
    Let $L$ be a finite-dimensional Lie algebra over $\mathbb{C}$. If $L$ is non-semisimple with trivial center and $\Gamma_L \cong \Gamma_{\mathfrak{g}}$ where $\mathfrak{g}$ is a non-abelian CT Lie algebra over $\mathbb{C}$, then $L$ is solvable.
\end{cor}
\begin{proof}
    Again, we have $\Gamma_L$ is a complete partite graph. Then by Theorem \ref{CTcomp}, $L$ is CT. Since $L$ is non-simple, we obtain that $L$ is solvable, by \cite[Theorem 4.1]{KM10}.
\end{proof}

Recall that a real compact Lie algebra is a Lie algebra of some compact Lie group. 
The results appeared in \cite{Gor17}, in particular, Lemma 1 and Lemma 2,  provide us some information about Lie algebras over $\mathbb{R}$ and we have:

\begin{cor}
    Let $L$ be a real semisimple Lie algebra with trivial center. Then 
    \begin{enumerate}
        \item if $L$ is compact and $\Gamma_L \cong \Gamma_{\mathfrak{su}(2)}$, then $L \cong \mathfrak{su}(2)$.
        \item if $L$ is noncompact and $\Gamma_L$ is isomorphic to $\Gamma_{\mathfrak{sl}_2(\mathbb{R})}$ or $\Gamma_{\mathfrak{sl}_2(\mathbb{C})}$, then $L \cong \mathfrak{sl}_2(\mathbb{R})$ or $L \cong \mathfrak{sl}_2(\mathbb{C})$.
    \end{enumerate}
\end{cor} 
\begin{proof}
    Applying Lemma 1 and Lemma 2 in \cite{Gor17}, 
    the proof can be done similarly to Corollary \ref{CTsl2}.
\end{proof}

\subsection{AC Lie algebras}

The following class of Lie algebras which is a generalization of a CT Lie algebra can be considered. A Lie algebra $L$ is said to be \textit{AC} if for all $x, y, z \in L \setminus Z(L)$, if 
$[x, y] = [y, z] = 0$, then $[x, z] = 0$. 

\begin{remark}
    If $L$ is an AC non-abelian Lie algebra, then $\gamma(\Gamma_L) \leq 2$. Indeed, since $L$ is non-abelian, there exist two adjacent vertices $[x], [y] \in \Gamma_L$. If $[z] \in \Gamma_L$ that $[z]$ is not adjacent to neither $[x]$ nor $[y]$, then $[x, z] = [y, z] = 0$ and so, $[x, y] =0$, a contradiction. 
\end{remark}

Note that any CT Lie algebra is an AC Lie algebra. However, the converse is not true in general unless the center is trivial. We give an example of an AC Lie algebra which is not a CT Lie algebra as follows.

\begin{ex}
    Let $L$ be a $3$-dimensional Heisenberg Lie algebra which the canonical commutation relations: $[x, y] = z$,
    where $x, y, z$ are the generators for $L$. It is known that $L$ is a nilpotent Lie algebra with $Z(L) = \spann \{z\}$. So, it is not a CT Lie algebra. Next, we show that $L$ is AC.

    Let $u, v, w \in L \setminus Z(L)$ such that $[u, v] = [v, w] = 0$. We write each one as a linear combination of all basis elements:
    \begin{align*}
        u &= a_1 x + b_1 y + c_1 z \\
        v &= a_2 x + b_2 y + c_2 z \\
        w &= a_3 x + b_3 y + c_3 z
    \end{align*}
where $a_i, b_i, c_i \in \mathbb{K}$ with $(a_i, b_i) \neq (0, 0)$ for all $i = 1, 2, 3$. Then the condition $[u, v] = [v, w] = 0$ implies $a_1 b_2 = a_2 b_1$ and $a_2 b_3 = a_3 b_2$. Without loss of generality, we may assume that $a_2 \neq 0$. Then $b_1 = \dfrac{a_1 b_2}{a_2}$ and $b_3 = \dfrac{a_3b_2}{a_2}$. Therefore, the bracket $[a, c] = (a_1 b_3 - a_3 b_1)z$ equals zero, and implying that $L$ is an AC Lie algebra.
\end{ex}

In the following lemma, we provide some equivalent definitions of AC Lie algebras.

\begin{lemma}\label{eqdefac}
    Let $L$ be a finite-dimensional Lie algebra over $\mathbb{K}$. Then the following statements are equivalent.
    \begin{enumerate}
        \item $L$ is an AC Lie algebra.
        \item For all $x \in L \setminus Z(L)$, $C_L(x)$ is abelian.
        \item For all $x, y \in L \setminus Z(L)$, if $[x, y] = 0$, then $C_L(x) = C_L(y)$.
        \item If $I$ and $J$ are subalgebras of $L$ and $Z(L) \subsetneq C_L(I) \subseteq C_L(J) \subsetneq L$, then $C_L(I) = C_L(J)$.
    \end{enumerate}
\end{lemma}
\begin{proof}
    The proof is straightforward as follows.

    To show that (1) implies (2), let $x \in L \setminus Z(L)$ and $u, v \in C_L(x)$. Then $[u, x] = [x, v] = 0$ and hence, by (1), $[u, v] = 0$. Now, we assume (2) and let $x, y \in L \setminus Z(L)$ such that $[x, y] = 0$. For any $u \in C_L(x)$, we have $[u, x] = 0$ So, both $u$ and $y$ are in $C_L(x)$. By (2), $C_L(x)$ is abelian, so then $[u, y] = 0$ and $u \in C_L(y)$. Next, we assume (3) and let $x, y, z \in L \setminus Z(L)$ such that $[x, y] = [y, z] = 0$ Then we have $C_L(x) = C_L(y) = C_L(z)$ and hence $[x, z] = 0$, we obtain (1) $\iff$ (2) $\iff$ (3).

    Suppose that $L$ is AC. Let $I$ and $J$ be subalgebras of $L$ such that $Z(L) \subsetneq C_L(I) \subseteq C_L(J) \subsetneq L$. Suppose that there exists $y \in C_L(J) \setminus C_L(I)$. Let $u \in I, v \in J \setminus Z(L)$ and $x \in C_L(I) \setminus Z(L)$. Then $[x, u] = [x, v] = 0$. By the assumption, $[u, v] = 0$. Moreover, $[v, y] = 0$ and so $[u, y] = 0$. Since $u$ is arbitrary, $y \in C_L(I)$, a contradiction and (4) is achieved. Finally, assume that (4) holds. Let $x \in L \setminus Z(L)$. We show that $C_L(x)$ is abelian. Let $y, z \in C_L(x)$ such that $[y, z] \neq 0$. Then $z \in C_L(\spann\{x, y\})$, and so $Z(L) \subsetneq C_L(\spann\{x, y\}) \subsetneq C_L(\spann\{x\}) \subseteq L$, which contradicts (4). This completes the proof.
\end{proof}

\begin{prop}\label{glnAC}
    The Lie algebra $\mathfrak{gl}_2(\mathbb{C})$ is AC whereas $\mathfrak{gl}_n(\mathbb{C})$ is not AC for all $n \neq 2$.
\end{prop}
\begin{proof}
    Recall that $\mathfrak{gl}_n(\mathbb{C}) \cong \mathfrak{sl}_n(\mathbb{C}) \oplus \lambda I_n$. Note that if $\mathfrak{gl}_n(\mathbb{C})$ is AC, then $\mathfrak{sl}_n(\mathbb{C})$ is CT.
    By \cite[Theorem 3.1]{KM10}, $\mathfrak{sl}_2(\mathbb{C})$ is the only CT Lie algebra among $\mathfrak{sl}_n(\mathbb{C})$. We have the Proposition.
\end{proof}

\begin{thm}\label{ACcomp}
    A finite-dimensional non-abelian Lie algebra $L$ is AC if and only if $\Gamma_L$ is a complete multipartite graph with each partite set is of the form $P(C_L(x)/Z(L))$ for some $x \in L \setminus Z(L)$.
\end{thm}
\begin{proof}
    The proof is similar to the proof of Theorem \ref{CTcomp} by replacing $\{0\}$ by $Z(L)$.
\end{proof}

\subsection{Lie algebras over finite fields}

We explore some further assertions about a non-commuting graph from a Lie algebra $L$ over a finite field $\mathbb{F}_q$.
When $L$ is AC, the graph $\Gamma_L$ has the following property.

\begin{thm}
    Let $L$ be a finite-dimensional non-abelian AC Lie algebra over $\mathbb{F}_q$. Then $\omega(\Gamma_L) = \chi(\Gamma_L)$.
\end{thm}
\begin{proof}
    Let $\omega = \omega(\Gamma_L)$ and $\chi = \chi(\Gamma_L)$. Suppose that $G$ is a maximal clique in $\Gamma_L$ containing vertices $[x_1], [x_2], \ldots, [x_\omega]$. Since $L$ is AC, by Lemma \ref{eqdefac}, each $C_L(x_i)$ is abelian. Moreover, by the maximality, we have $L = \cup_{i = 1}^\omega C_L(x_i)$. By Theorem \ref{chromatic}, $\chi \leq \omega$. Note that the converse of this always holds for any graph. 
\end{proof}

\begin{rmk}
      Let $L$ be a finite-dimensional non-abelian Lie algebra over $\mathbb{K}$. If $\Gamma_L \cong \Gamma_{\mathfrak{g}}$ for some finite-dimensional Lie algebra $\mathfrak{g}$ over a finite field and $\Gamma_{\mathfrak{g}}$ contains at least two vertices, then by Proposition \ref{basic}, $\mathbb{K}$ is a finite field.  
\end{rmk}

It is worth noting that Question \ref{isomq1} may be highly non-trivial in general. However, for finite graphs, a weaker variant of the question arises naturally:

\begin{quest}\label{isomq2}
    Let $\mathfrak{g}$ and $\mathfrak{h}$ be two finite-dimensional Lie algebras over finite fields such that $\Gamma_{\mathfrak{g}} \cong  \Gamma_{\mathfrak{h}}$. Is it true that $|\mathfrak{g}| = |\mathfrak{h}|$?
\end{quest}

It is clear that a positive answer to this question can be obtained when the two Lie algebras are defined over the same finite field, as shown in the following theorem.

\begin{thm}
    Let $L$ and $\mathfrak{g}$ be two finite-dimensional non-abelian Lie algebras over $\mathbb{F}_{p^{m}}$, where $p$ is a prime and $m \in \mathbb{N}$. If $\Gamma_L \cong \Gamma_{\mathfrak{g}}$, then $|L| = |\mathfrak{g}|$ if and only if $|Z(L)| = |Z(\mathfrak{g})|$.
\end{thm}
\begin{proof}
    Let $d_1 = \dim(L/Z(L))$ and $d_2 = \dim(\mathfrak{g}/Z(\mathfrak{g}))$.
     Now, assume that $\Gamma_L \cong \Gamma_{\mathfrak{g}}$. Then $|P(L/Z(L))| = |P(\mathfrak{g}/Z(\mathfrak{g}))|$ and so, $\dfrac{p^{m d_1} - 1}{p^m - 1} = \dfrac{p^{m d_2} - 1}{p^m - 1}$. Thus, $d_1 = d_2$, and the result follows.
\end{proof}

\section{Concluding Remark}
In Section \ref{ISO}, we have investigated certain classes of Lie algebras in which two Lie algebras possess the same non-commuting graphs. For future development, one may try to analyze the graph or respond to Question \ref{isomq1} and Question \ref{isomq2} for other types of Lie algebras, such as (finite) nilpotent or simple Lie algebras. 

We observe that AC Lie algebras have a direct relationship with their non-commuting graphs. To gain deeper insights, it is imperative to delve into the theory of AC Lie algebras. However, as far as our knowledge extends, there are no references that specifically mention AC Lie algebras. This type of Lie algebra requires further attention in its own right.

\section*{Acknowledgement}
\noindent 
This project is funded by National Research Council of Thailand (NRCT) and King Mongkut's University of Technology Thonburi (N42A680277).


\begin{thebibliography}{99}

\bibitem{AAM06}
A. Abdollahi, S. Akbari, H. R. Maimani, Non-commuting graph of a group. J. Algebra, 298 (2006) 468-492.

\bibitem{B83}
E. A. Bertram, Some applications of graph theory to finite groups, Discrete Math., 44 (1983) 31--43.

\bibitem{BHM90}
E. A. Bertram, M. Herzog, A. Mann, On a graph related to conjugacy classes of groups, Bull. London Math. Soc., 22 (6) (1990) 569--575.

\bibitem{BI03}
Y. Barnea, I.M. Isaac, Lie algebras with few centralizer dimensions. J. Algebra, 259 (2003) 284--299.

\bibitem{CES24}
A. Chareh Khan Moghaddam, A. Erfanian, A. Shamsaki, On the non-commuting graph associated to a finite-dimensional Lie algebra, preprint: 	arXiv:2406.15902.

\bibitem{CM09}
S. M. Cioabă, R. M. Maruty, A first course in graph theory and combinatorics. Vol. 55, New Delhi, India: Hindustan Book Agency, Springer, 2009.

\bibitem{Dirac52}
G.A. Dirac, Some theorems on abstract graphs, Proc. London Math. Soc., 3 (1952) No. 2, 69--81. doi:10.1112/plms/s3-2.1.69. 

\bibitem{EW06}
K. Erdmann, M. J. Wildon, Introduction to Lie algebras. Vol. 122. London: Springer, 2006.


\bibitem{F25}
S. D. Freedman, The non-commuting, non-generating graph of a finite simple group, Q. J. Math., 76 (2025) Issue 1 313--335.

\bibitem{GR01}
C. Godsil, G. Royle, Algebraic graph theory, New York, Springer, 2001.

\bibitem{Gor17}
V. V. Gorbetsevich, Lie algebras with abelian centralizers, Mathematical Note, 101 (2017) No. 5 795--801.

\bibitem{GM23}
V. Grazian, C. Monetta, A conjecture related to the nilpotency of groups with isomorphic non-commuting graphs, J. Algebra, 633 (2023) 389–402.

\bibitem{H12}
J. E. Humphreys, Introduction to Lie algebras and representation theory. Vol. 9. Springer Science \& Business Media, 2012.

\bibitem{KM10}
I. Klep, P. Moravec, Lie algebras with abelian certralizers. Algebra Colloq., 4 (2010) 629--636.

\bibitem{KNS23}
R. Kundu, T. K. Naik, A. Singh, Nilpotent Lie algebras with two centralizer dimensions over a finite field. J. Algebra, 633 (2023) 362--388.

\bibitem{LW81}
J. C. Lennox, J. Wiegold, Extensions of a problem of Paul Erd\"os on groups, J. Aust. Math. Soc. Ser. A, 31 (1981) 459--463.

\bibitem{N76}
B. H. Neumann, A problem of Paul Erd\"os on groups, J. Aust. Math. Soc. Ser. A, 21 (1976) 467--472.

\bibitem{S24}
H. Shahverdi, Finite groups with isomorphic non-commuting graphs have the same nilpotency property, J. Algebra, 642 (2024) 60-64.

\end{thebibliography}
\end{document}